\documentclass[11pt,a4paper]{amsart}
\usepackage{amsmath,amssymb, amsbsy}
\usepackage{color,psfrag}
\usepackage[dvips]{graphicx}
\numberwithin{equation}{section}
\usepackage{enumerate}
\newcommand{\R}{{\mathbb R}}
\newcommand{\N}{{\mathbb N}}

\renewcommand{\geq }{\geqslant}

\renewcommand{\leq }{\leqslant}
\def\neweq#1{\begin{equation}\label{#1}}
\def\endeq{\end{equation}}
\def\eq#1{(\ref{#1})}
\newtheorem{theorem}{Theorem}[section]
\newtheorem{proposition}[theorem]{Proposition}
\newtheorem{lemma}[theorem]{Lemma}

\theoremstyle{definition}

\begin{document}
	
	\title[Non-homogeneous partially hinged plates]{Optimization of the structural performance of non-homogeneous  partially hinged rectangular plates}

\author[Alessio FALOCCHI]{Alessio FALOCCHI}
\address{\hbox{\parbox{5.7in}{\medskip\noindent{Dipartimento di Scienze Matematiche, \\
				Politecnico di Torino,\\ Corso Duca degli Abruzzi 24, 10129 Torino, Italy. \\[3pt]
				\em{E-mail address: }{\tt alessio.falocchi@polito.it}}}}}
				\keywords{gap function; torsional instability; mass density}
				
				\begin{abstract} 
					We consider a non-homogeneous partially hinged rectangular plate having structural engineering applications. In order to study possible remedies for torsional instability phenomena we consider the gap function as a measure of the torsional performances of the plate. We treat different configurations of load and we study which density function is optimal for our aims. The analysis is in accordance with some results obtained studying the corresponding eigenvalue problem in terms of maximization of the ratio of specific eigenvalues. Some numerical experiments complete the analysis.
					
				\end{abstract}
			\maketitle
				
		\section{Introduction}\label{1}	
		We study a long narrow rectangular thin plate $\Omega\subset \R^2$, hinged at the short edges and free on the remaining two, see \cite{fergaz}. This plate may model the deck of a bridge; since this kind of structure exhibits problems of flutter instability, e.g. see \cite{bookgaz,larsen,rocard}, we optimize its design in order to reduce the phenomenon. To this aim one may vary the shape of the plate, see \cite{bebuga2}, or modify the materials composing it, see \cite{bebugazu,befafega,befa}. \par 
		Here we fix the geometry of the plate, assuming that it has length $\pi$ and width $2\ell$ with $2\ell\ll\pi$ so that
		$$
		\Omega=(0,\pi)\times(-\ell,\ell)\subset\R^2\,;
		$$ 
		we assume that  the plate is not homogeneous, i.e. it features variable density function $p=p(x,y)$; our aim is to find the optimal density configuration in order to improve the structural performance of the plate.
		
		In a rectangular plate it is possible to distinguish vertical and torsional oscillations; the most problematic are the second ones, that may cause the collapse of the structure, see \cite{bookgaz}.
		Then we consider a functional, named \textit{gap function}, able to
		measure the torsional performance of the plate, see also \cite{proceedingpadova}. In particular, this functional
		measures the gap between the displacements of the two free edges of the structure; the higher is the gap the higher is the torsional motion of the plate. More precisely, we maximize the maximum of the absolute value of the gap function in a class of external forcing term; then we consider its minimization in a class of density functions. Hence, our final goal is to find the worst force and the best density in order to reduce the torsional oscillation of the plate. 
		
		Since the explicit solution of this \textit{minimaxmax} problem is currently out of reach, we proceed testing the plate with some motivated external forces.
		Then we consider different densities $p(x,y)$ in order to understand how the gap function varies; the choice of $p(x,y)$ is driven by some results proposed in \cite{befa}.
		 Here the authors present a study on the correspondent weighted eigenvalue problem and they compare different density functions in order to find the optimal, maximizing the ratio between the first torsional eigenvalue and the previous longitudinal; they tested some density functions proposing theoretical and numerical justifications. We point out that the study of a ratio of eigenvalues has some limits; first of all it requires to consider two specific eigenvalues, moreover the direct optimization of the ratio is very involved. As a consequence, the question is often dealt with in terms of minimization or maximization of a single eigenvalue, see \cite{befa} for details.  Here we compare the density functions proposed in \cite{befa} and we observe that $p(x,y)$ optimal for \cite{befa} are optimal also with respect to the reduction of the gap function. This result confirms that the gap function is a reliable measure for the torsional performances of rectangular plates; furthermore, it is a useful tool to get information on optimal reinforces in order to reduce torsional instability phenomena.\par The paper is organized as follows. In Section \ref{prelim} we introduce some preliminaries and notations and we define longitudinal and torsional modes of vibration. In Section \ref{gapsec} we define the gap function, we write the \textit{minimaxmax} problem we are interested in and we state the existence results, proved in Section \ref{proof}. In Section \ref{weightsec} we describe the density functions that are meaningful for our aims. In Section \ref{L2sec} we study the problem considering external forces in $L^2(\Omega)$ and providing some numerical experiments to support the theoretical results. 
		\par

		\section{Preliminaries and Variational setting}\label{prelim}
		\subsection{Definition of the problem}
		We derive the stationary equation which we are interested in from the energy of the system; we denote by $u=u(x,y)$ the vertical displacement of the plate $\Omega$ having mass surface density $p=p(x,y)$. In general, since we are dealing with a non-homogeneous plate, we may consider the modulus of Young $E=E(x,y)$ and the Poisson ratio $\sigma=\sigma(x,y)$ of the materials forming the plate not constant. 
		We suppose that an external force for the unit mass $f=f(x,y)$ acts on the plate in the vertical direction. Thanks to the Kirchhoff-Love theory \cite{Kirchhoff,Love}, the energy of the plate is given by
		\begin{equation*}\label{energy}
		\mathbb{E}(u)=\dfrac{h^3}{12}\int_{\Omega}\dfrac{E}{1-\sigma^2}\bigg(\dfrac{(\Delta u)^2}{2}+(1-\sigma)(u_{xy}^2-u_{xx}u_{yy})\bigg)dxdy-\int_{\Omega}pfu\,dxdy,
		\end{equation*}
		where $h$ is its constant thickness, see also \cite{fergaz}.\par 
		To proceed with the classical minimization of the functional, we need some information on the regularity of the functions representing the materials composing the plate, i.e. $p(x,y)$, $E(x,y)$, $\sigma(x,y)$. We consider the possibility that the plate is composed by different materials, hence we cannot assume the continuity of the previous functions. In general discontinuous Young modulus and Poisson ratio generate some mathematical troubles in finding the minimization problem in strong form. For the civil engineering applications, which we are interested in, we point out that the Poisson ratio does not vary so much with respect to the possible choice of the materials; therefore, as a first approach, we suppose $E$ and $\sigma$ constant in space, while the density of the plate is in general variable and possibly discontinuous. Hence we have
		 \begin{equation*}\label{energy2}
		 \mathbb{E}(u)=\dfrac{Eh^3}{12(1-\sigma^2)}\int_{\Omega}\bigg(\dfrac{(\Delta u)^2}{2}+(1-\sigma)(u_{xy}^2-u_{xx}u_{yy})\bigg)dxdy-\int_{\Omega}pfu\,dxdy;
		 \end{equation*}
		in this framework we minimize the energy functional, we divide the differential equation for the flexural rigidity $\frac{Eh^3}{12(1-\sigma^2)}$ and, including it in the density function, we obtain
			\begin{equation}\label{weight}
		\begin{cases}
		\Delta^2 u=p(x,y) f(x,y) & \qquad \text{in } \Omega \\
		u(0,y)=u_{xx}(0,y)=u(\pi,y)=u_{xx}(\pi,y)=0 & \qquad \text{for } y\in (-\ell,\ell)\\
		u_{yy}(x,\pm\ell)+\sigma
		u_{xx}(x,\pm\ell)=u_{yyy}(x,\pm\ell)+(2-\sigma)u_{xxy}(x,\pm\ell)=0
		& \qquad \text{for } x\in (0,\pi)\, .
		\end{cases}
		\end{equation}
		The boundary conditions on the short edges are of Navier type, see \cite{navier}, and model the situation in which the plate is hinged on $\{0,\pi\}\times(-\ell,\ell)$. Instead, the boundary conditions on the large edges are of Neumann type, modeling the fact that the deck is free to move vertically; for the Poisson ratio we shall assume 
		\begin{equation}\label{poisson}
		\sigma\in\bigg(0,\dfrac{1}{2}\bigg),
		\end{equation}
		since most of the materials have values in this range.\par 
		In the sequel we denote by $\|\cdot\|_q$ the norm related to the Lebesgue spaces $L^q(\Omega)$ with $1\leq q\leq\infty$ and  we refer to $q'$ as the conjugate of $q$, i.e. $1/q+1/q'=1$ with the usual conventions; moreover, given a functional space $V(\Omega)$, in the notation of the correspondent norm and scalar product we shall omit the set $\Omega$, e.g. $\|\cdot\|_{V}:=\|\cdot\|_{V(\Omega)}$. 
		
		In the next sections we study the behaviour of the plate with respect to different weight functions $p$ and external forcing terms $f$.

		\subsection{Families of forcing terms and weight functions}

		We introduce $$\mathcal{F}_{V}:=\{f\in V(\Omega): \|f\|_{V}=1\}$$ the set of admissible forcing terms, fixed a certain functional space $V$.
		We introduce a family of weights to which $p$ belongs 
		\begin{equation} \label{eq:famiglia}
		\mathcal{P}_{L^\infty}^{\alpha, \beta}:=\left\{p\in L^\infty(\Omega):\,\alpha\leq p\leq\beta\,,\,\,  p(x,y)=p(x,-y)  \ \text{a.e. in } \Omega,\,\, \int_{\Omega}p\,dxdy=|\Omega| \, \right\} \, 
		\end{equation} 	
where $\alpha,\beta \in\mathbb{R}^+$ with $\alpha<\beta$ fixed.  	
When $f$ belongs to certain functional spaces, we need further regularity on the weight functions; therefore we introduce a second family
\begin{equation*} \label{eq:famiglia2}
\mathcal{P}_{H^2}^{\alpha, \beta}:=\left\{p\in H^2(\Omega):\,p\in \mathcal{P}_{L^\infty}^{\alpha, \beta}\quad\text{and}\quad\exists\,\kappa>1\,:\ \|p\|_{H^2}\leq \kappa\sqrt{|\Omega|} \right\},
\end{equation*}
with $\alpha,\beta \in\mathbb{R}^+$ and $\alpha<\beta$ fixed.
 The integral condition in \eqref{eq:famiglia} represents the preservation of the total mass of the plate; this is our fixed parameter, useful to compare the results between different weights.  The bound on $\|p\|_{H^2}$ in $\mathcal{P}_{H^2}^{\alpha, \beta}$ is merely a technical condition to gain compactness; by H\"{o}lder inequality the preservation of the total mass condition yields $\|p\|_{H^2}\geq\sqrt{|\Omega|} $. Therefore, we choose  $\kappa>1$ to exclude the trivial case $p\equiv1$ in $\Omega$.
 Indeed, we will always assume $$0<\alpha<1<\beta\,,$$
 studying the effect of a non-constant weight on the solution of \eq{weight}.	
 The assumption $\alpha<1<\beta$ is not restrictive; if we assume $\beta= 1$, it must be $p \equiv 1$ a.e. in $\Omega$, since otherwise we would have
 $\int_\Omega p\,dx\,dy<|\Omega|$; similarly, if we consider $\alpha=1$.\par
Moreover, we are interested in designs which are symmetric with respect to the mid-line of the roadway, being $\ell$ very small with respect to $\pi$. From a mathematical point of view, this assures two classes of eigenfunctions for the correspondent eigenvalue problem, respectively, even or odd in the $y$-variable; we shall clarify this question in Section \ref{torsdef}. \par
	 \subsection{Existence and uniqueness result}
	  We introduce the space 
	 $$
		H^2_*(\Omega)=\big\{u\in H^2(\Omega): u=0\mathrm{\ on\ }\{0,\pi\}\times(-\ell,\ell)\big\}\,,
		$$
		where we study the weak solution of \eqref{weight}.
		Let us observe that the condition $u=0$ has to be meant in a classical sense because $\Omega\subset\R^2$ and the energy space $H^2_*(\Omega)$ embeds into continuous functions. 
		Furthermore, $H^2_*(\Omega)$ is a Hilbert space when endowed with the scalar product
		$$
		(u,v)_{H^2_*}:=\int_\Omega \left[\Delta u\Delta v+(1-\sigma)(2u_{xy}v_{xy}-u_{xx}v_{yy}-u_{yy}v_{xx})\right]\, dx \, dy \,
		$$
		and associated norm
		$$
		\|u\|_{H^2_*}^2=(u,u)_{H^2_*} \, ,
		$$
		which is equivalent to the usual norm in $H^2(\Omega)$, see \cite[Lemma 4.1]{fergaz}. We denote by $H_*^{-2}(\Omega)$ the dual space of $H_*^2(\Omega)$ and $\langle\cdot,\cdot\rangle$ its dual product.
		We write the problem \eq{weight} in weak sense
		\begin{equation}
		\label{eigenweak}
		(u,v)_{H^2_*} =\langle pf,v\rangle \qquad\forall v\in H^2_*(\Omega).
		\end{equation}
		Let us clarify what we mean for the dual product in \eqref{eigenweak} with respect to the choice of $f$ and $p$. 
		
		If $f\in \mathcal{F}_{L^q}$ with $q\in(1,\infty]$ and $p\in \mathcal{P}_{L^\infty}^{\alpha,\beta}$, we write $\int_{\Omega}pf v\,dxdy$ instead of $\langle pf,v\rangle$.

		If $f\in H^{-2}_*(\Omega)$ we need further regularity on $p$, e.g. $p\in \mathcal{P}_{H^2}^{\alpha,\beta}$. 	We introduce the linear functional $T_f: H^2_*(\Omega)\rightarrow \mathbb{R}$ such that  $T_f(v):=\langle f,v\rangle$ for all $v\in H^2_*(\Omega)$ and we define 
		\begin{equation}\label{defdual}
\langle pf,v\rangle:=T_f(pv)\qquad  \forall v\in H^2_*(\Omega).
		\end{equation}
		
		Indeed, $H^2_*(\Omega)$ is a Banach algebra, being the  $H^2_*(\Omega)$-norm equivalent to the $H^2(\Omega)$-norm, see \cite[Theorem 5.23]{adams} applied to the Sobolev space $W^{m,p}(\Omega)$ with $m=p=2$ and $\Omega\subset \mathbb{R}^2$ convex with Lipschitz boundary. Therefore, if $p\in \mathcal{P}_{H^2}^{\alpha, \beta}$ we get $K>0$ such that
		\begin{equation*}
		pv\in H^2_*(\Omega)\qquad	\|pv\|_{H^2_*}\leq K\|p\|_{H^2_*}\|v\|_{H^2_*}\qquad \forall v\in H^2_*(\Omega).
		\end{equation*}
	We state the following result.
		\begin{proposition}\label{uniqueness}
			Let $f\in\mathcal{F}_V$ and $0<\alpha<1<\beta$. If
			\begin{enumerate}[i)]
				\item $V=L^q(\Omega)$ with $q\in (1,\infty]$ and $p\in \mathcal{P}_{L^\infty}^{\alpha, \beta}$,
				\item $V=H^{-2}_*(\Omega)$ and $p\in \mathcal{P}_{H^2}^{\alpha, \beta}$ ,
			\end{enumerate} 
			then the problem \eqref{eigenweak} admits a unique weak solution $u\in H^2_*(\Omega)\subset C^0(\overline \Omega)$.
		\end{proposition}

	\begin{proof}	
	By \cite{fergaz} we have that the bilinear form $(u,v)_{H^2_*}$ is continuous and coercive,  hence to apply Lax Milgram Theorem we consider the functional $\langle pf,v\rangle$. 
	
	$i)$ If $p\in\mathcal{P}_{L^\infty}^{\alpha,\beta}$ and $f\in \mathcal{F}_{L^q}$ with $q\in (1,\infty]$ then $pf\in L^q(\Omega)$; moreover we have $\Omega\subset \mathbb{R}^2$ so that $H^2_*(\Omega)$ is embedded in $C^0(\overline \Omega)$. Therefore, applying H\"older inequality, we obtain $C_1>0$ such that
	$$
	|\langle pf,v\rangle|=\bigg|\int_\Omega pfv\,dxdy\bigg|\leq \|pf\|_q\|v\|_{q'}\leq C_1\|v\|_{H^2_*}\qquad \forall v\in H^2_*(\Omega),
	$$
		so that $\langle pf,v\rangle$ is a linear and continuous functional.

$ii)$ By \eqref{defdual} we observe that $T_f(pv)$ is linear and continuous, indeed we have $C_2>0$ such that
$$
|T_f(pv)|=|\langle f,pv\rangle|\leq\|f\|_{H^{-2}_*} \|pv\|_{H^2_*}\leq C_2\|v\|_{H^2_*}\qquad \forall v\in H^2_*(\Omega),
$$	
being $H^2_*(\Omega)$ a Banach algebra.

The solution $u$ is continuous since the space $H^2_*(\Omega)$ embeds into $C^0(\overline \Omega)$.
\end{proof}

\subsection{Definition of longitudinal and torsional modes}\label{torsdef}
 To tackle \eq{weight} we need some preliminary information on the associated eigenvalue problem:
\begin{equation}\label{weighteig}
\begin{cases}
\Delta^2 u=\lambda p(x,y)u & \qquad \text{in } \Omega \\
u(0,y)=u_{xx}(0,y)=u(\pi,y)=u_{xx}(\pi,y)=0 & \qquad \text{for } y\in (-\ell,\ell)\\
u_{yy}(x,\pm\ell)+\sigma
u_{xx}(x,\pm\ell)=u_{yyy}(x,\pm\ell)+(2-\sigma)u_{xxy}(x,\pm\ell)=0
& \qquad \text{for } x\in (0,\pi)\, .
\end{cases}
\end{equation}
	As in \cite{bogamo}, we introduce the subspaces of $H^2_*(\Omega)$:
\begin{equation*}\label{subspaces}
\begin{split}
&H^2_\mathcal{E}(\Omega):=\{u\in H^2_*(\Omega): u(x,-y)=u(x,y)\quad\forall (x,y)\in\Omega\},\\&H^2_\mathcal{O}(\Omega):=\{u\in H^2_*(\Omega): u(x,-y)=-u(x,y)\quad\forall (x,y)\in\Omega\},
\end{split}
\end{equation*}
where
\begin{equation}\label{decomposition}
H^2_\mathcal{E}(\Omega)\perp H^2_\mathcal{O}(\Omega), \hspace{5mm}H^2_*(\Omega)=H^2_\mathcal{E}(\Omega)\oplus H^2_\mathcal{O}(\Omega)\,.
\end{equation}
We say that the eigenfunctions in $H^2_\mathcal{E}(\Omega)$ are \emph{longitudinal} modes and those in $H^2_\mathcal{O}(\Omega)$ are \emph{torsional} modes.  For all $u\in H^2_*(\Omega)$ we denote by $u^{e}=\frac{u(x,y)+u(x,-y)}{2}\in H^2_\mathcal{E}(\Omega)$ and $u^o=\frac{u(x,y)-u(x,-y)}{2}\in H^2_\mathcal{O}(\Omega)$ respectively its even and odd components. Moreover, we set
\begin{equation*}\label{dualsubspaces}
\begin{split}
&H^{-2}_\mathcal{E}(\Omega):=\{f\in H^{-2}_*(\Omega): \langle f,v\rangle=0\quad\forall v\in H^2_\mathcal{O}(\Omega)\},\\&H^{-2}_\mathcal{O}(\Omega):=\{f\in H^{-2}_*(\Omega): \langle f,v\rangle=0\quad\forall v\in H^2_\mathcal{E}(\Omega)\}.
\end{split}
\end{equation*}
Since $H^2_*(\Omega)=H^{-2}_\mathcal{E}(\Omega)\oplus H^{-2}_\mathcal{O}(\Omega)$,  there exists a unique couple $(f^e,f^o)\in H^{-2}_\mathcal{E}(\Omega)\times H^{-2}_\mathcal{O}(\Omega)$ such that $f=f^e+f^o$ for all $f\in H^{-2}_*(\Omega)$.
We endow the space $H^{-2}_*(\Omega)$ with the norm  $\|f\|_{H^{-2}_*}:=\sup_{\|v\|_{H^2_*}=1 }\langle f,v\rangle$, observing that
\begin{equation}\label{dualnorm}
\|f\|_{H^{-2}_*}=\max\{\|f^o\|_{H^{-2}_*},\|f^e\|_{H^{-2}_*}\}	\qquad \forall f\in H^{-2}_*(\Omega).
\end{equation}

When  $p\equiv 1$ the whole spectrum of \eqref{weighteig} is determined explicitly in \cite{fergaz} and gives two class of eigenfunctions belonging respectively to $H^2_\mathcal{E}(\Omega)$ or $H^2_\mathcal{O}(\Omega)$. Thanks to the symmetry assumption on $p$ we obtain the same distinction for all the linearly independent eigenfunctions of the weighted eigenvalue problem \eq{weighteig}.


  We denote by $\mu_m(p)$ and $\nu_m(p)$ respectively the ordered weighted longitudinal and torsional eigenvalues of \eq{weighteig}, repeated with their multiplicity; moreover, we denote respectively by $z^p_{m}(x,y)\in H^2_\mathcal{E}(\Omega)$ and $\theta^p_m(x,y)\in H^2_\mathcal{O}(\Omega)$, the corresponding (ordered) longitudinal and torsional linearly independent eigenfunctions of \eqref{weighteig}. We consider the eigenfunctions normalized in $L^2_p(\Omega)$ ($L^2(\Omega)$-weighted), i.e. 
\begin{equation}\label{normalization}
\|\sqrt{p}\,z^p_m\|^2_2=\int_{\Omega}p\,(z^p_m)^2\,dxdy=1\qquad \|\sqrt{p}\,\theta^p_m\|^2_2=\int_{\Omega}p\,(\theta^p_m)^2\,dxdy=1.
\end{equation}


\section{Gap function}\label{gapsec}
In real structures the most problematic motions are related to the torsional oscillations, i.e. those in which prevail torsional modes. How can we measure the torsional behaviour? By Proposition \ref{uniqueness}, the solution of \eq{weight} is continuous; hence, we define the \textit{gap function}, see also \cite{proceedingpadova},
\begin{equation}\label{gapdef}
\mathcal{G}_{f,p}(x):=u(x,\ell)-u(x,-\ell)\qquad \forall x\in[0,\pi],
\end{equation}
depending on the weight $p$ and on the external load $f$.
This function gives for every $x\in[0,\pi]$ the difference between the vertical displacements of the free edges, providing a measure of the torsional response. The maximal gap is given by
\begin{equation}\label{gapmax}
\mathcal{G}^{\infty}_{f,p}:=\max\limits_{x\in(0,\pi)}|\mathcal{G}_{f,p}(x)|.
\end{equation}
In this way we introduce the map $\mathcal{G}^\infty_{f,p}:\mathcal{F}_V\times \mathcal{P}_W^{\alpha, \beta}\rightarrow [0,+\infty)$ with $(f,p)\mapsto\mathcal{G}^\infty_{f,p}$, that we study respectively in the cases
\begin{equation}\label{casi}
\begin{split}
i)&\quad(V,W)=\big(L^q(\Omega),L^\infty(\Omega)\big) \text{ with } q\in(1,\infty]\\
ii)&\quad (V,W)=\big(H^{-2}_*(\Omega),H^2(\Omega)\big)
\end{split}
\end{equation}
for which Proposition \ref{uniqueness} assures the uniqueness of a solution to \eqref{weight}.\par 
Our aim is to find the worst $f\in\mathcal{F}_V$, i.e. the forcing term that maximizes $\mathcal{G}^\infty_{f,p}$, and the best weight $p\in \mathcal{P}_W^{\alpha, \beta}$ that minimizes $\mathcal{G}^\infty_{f,p}$.
More precisely we want to solve the \textit{minimaxmax} problem
\begin{equation*}
\mathcal{G}^\infty:=\min\limits_{p\in \mathcal{P}_W^{\alpha, \beta}}\max\limits_{f\in\mathcal{F}_V}\max\limits_{x\in(0,\pi)}|\mathcal{G}_{f,p}(x)|,
\end{equation*}
in the cases \eqref{casi}.\par 

In Section \ref{proof} we prove the existence results.
\begin{theorem}\label{existence}
	Given $p\in \mathcal{P}_W^{\alpha, \beta}$ with $0<\alpha<1<\beta$, if
	\begin{enumerate}[i)]
		\item $W=L^\infty(\Omega)$ and $f\in\mathcal{F}_{V}$ with $V=L^q(\Omega)$ $q\in (1,\infty]$,
		\item $W=H^2( \Omega)$ and $f\in\mathcal{F}_{V}$ with $V=H^{-2}_*(\Omega)$,
	\end{enumerate} 
	then the problem 
	\begin{equation}\label{max}
	\mathcal{G}^\infty_p:=\max\limits_{f\in\mathcal{F}_{V}}\mathcal{G}^\infty_{f,p}
	\end{equation}
	admits solution.\par 
\end{theorem}
\begin{theorem}\label{existence2}
	Given $f\in \mathcal{F}_{V}$, if
	\begin{enumerate}[i)]
		\item $V=L^q(\Omega)$ with $q\in (1,\infty]$ and $p\in\mathcal{P}_W^{\alpha,\beta}$ ($0<\alpha<1<\beta$) with $W=L^\infty(\Omega)$,
		\item $V=H^{-2}_*(\Omega)$ and $p\in\mathcal{P}_W^{\alpha,\beta}$ ($0<\alpha<1<\beta$) with $W=H^2(\Omega)$,
	\end{enumerate} 
	then the problem 
	\begin{equation}\label{min}
	\min\limits_{p\in \mathcal{P}_W^{\alpha,\beta}}\mathcal{G}^\infty_p,
	\end{equation}
	
	admits solution.
	
\end{theorem}

The next result shows that for $p\in\mathcal{P}^{\alpha,\beta}_W$ ($y$-even), the worst force $f\in\mathcal{F}_V$ in terms of torsional performance can be sought in the class of the $y$-odd distributions or functions. 

\begin{proposition}\label{symteo}
	$i)$ Let $(V,W)$ as in \eqref{casi}-$i)$ then problem \eq{max} is equivalent to
	$$
	\max\{\mathcal{G}^\infty_{f,p}: f\in\mathcal{F}_{L^q},\,\, f(x,-y)=-f(x,y) \text{ a.e. in } \Omega\}.
	$$
	Moreover, if $q\in(1,\infty)$ any maximizer is necessarily odd with respect to $y$.
	
	$ii)$ Let $(V,W)$ as in \eqref{casi}-$ii)$, then problem \eq{max} is equivalent to
	$$
	 \max\{\mathcal{G}^\infty_{f,p}: f\in H^{-2}_\mathcal{O},\,\, \|f\|_{H^{-2}_*}=1\}.
	$$	
\end{proposition}
This proposition and its proof are inspired by \cite[Theorem 4.1-4.2]{bebugazu}, where a similar problem is dealt with and further results are given. We underline that the uniqueness of a $y$-odd maximizer is not guaranteed; indeed, solely in the case \eqref{casi}-$i)$ with $q\in(1,\infty)$ we obtain \textit{only} odd maximizers. In the cases \eqref{casi}-$i)$ with $q=\infty$ and \eqref{casi}-$ii)$ it is possible that other $f$, not necessarily odd, attain the maximum, see also \cite{bebugazu}.

	 \section{The choice of the weight function $p\in \mathcal{P}_{L^\infty}^{\alpha,\beta}$}\label{weightsec}
	  About the choice of the weight function $p\in \mathcal{P}_{W}^{\alpha,\beta}$ we are mainly interested in density functions not necessarily continuous, hence we consider $W=L^\infty(\Omega)$; therefore, in the rest of the paper we focus on \eqref{max}-\eqref{min} in the case $(V,W)=\big(L^q(\Omega),L^\infty(\Omega)\big)$ with $q\in(1,\infty]$. 
	
	 We refer to some results obtained on the correspondent eigenvalue problem \eqref{weighteig} presented in \cite{befa}.
Here the authors find the best rearrangement of materials in $\Omega$ which maximizes the ratio between two selected eigenvalues of \eqref{weighteig}, considering the optimization problem:
	 \begin{equation}\label{opt_intro}
	 \mathcal{R}=\sup_{p\in \mathcal{P}_{L^\infty}^{\alpha,\beta}}\dfrac{\nu(p)}{\mu(p)},
	 \end{equation}
	 where $\nu(p)$ and $\mu(p)$ are respectively a torsional and a longitudinal eigenvalue.
	 The direct study of \eq{opt_intro} is very involved, then there are some theoretical results on the problem of maximization of the first torsional eigenvalue or minimization of the first  longitudinal eigenvalue with respect to $p$; these results give suggestions on \eq{opt_intro} and support some conjectures also thanks to numerical experiments. More precisely, in \cite{befa} the authors proved theoretically that optimal weights $p(x,y)$ in increasing or reducing the first torsional or longitudinal eigenvalue must be of \textit{bang-bang} type, i.e.
	 $$
	 p(x,y) = \alpha \chi_{S} (x,y)+ \beta \chi_{\Omega \setminus S}(x,y)\,\quad \text{for a.e. } (x,y)\in \Omega\,,
	 $$
	 for a suitable set $S\subset \Omega$, $0<\alpha<1<\beta$ and $\chi_S$ is the characteristic function of $S$. In other words, the plate must be composed by two different materials properly located in $\Omega$; this is useful in engineering terms, since the manufacturing of two materials with constant density is simpler than the  assemblage of a material having variable density. On the other hand this produces some mathematical troubles, for instance when we consider as external forcing term $f\in H^{-2}_*(\Omega)$, see Proposition \ref{uniqueness}. \par 
	 In the sequel we distinguish five meaningful \textit{bang-bang} configurations  for $p\in \mathcal{P}_{L^\infty}^{\alpha,\beta}$; we list the cases representing on the right in black the localization of the reinforcing material on the plate:
	 \begin{enumerate}[i)]
	 	\item \fbox{$p\equiv 1$} \hfill\includegraphics[scale=0.22]{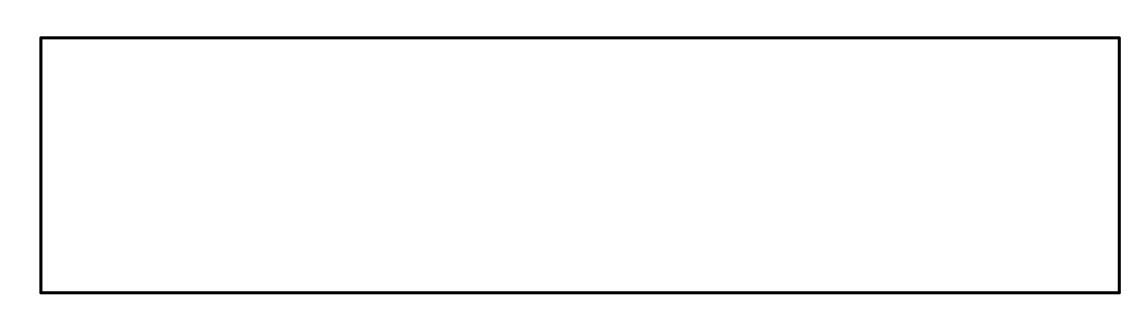}\vspace{2mm} \\
	 	This is a particular case when $\alpha=\beta=1$ that corresponds to the homogeneous plate; we do not apply reinforcements, but we consider this case to compare it with the non-homogeneous ones.\vspace{2mm}
	 	\item \fbox{$
	 		p^*(x,y)$}  \hfill\includegraphics[scale=0.22]{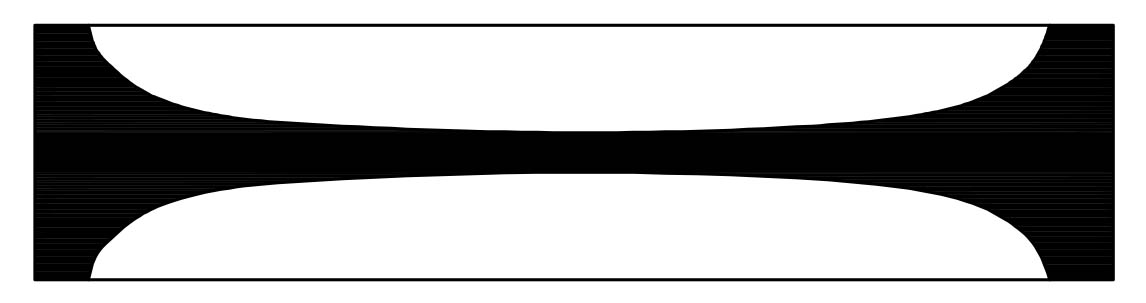}\vspace{2mm} \\
	 	This choice comes out from the study of the problem 
	 	\begin{equation}\label{CP2}
	 	\nu^{\alpha,\beta}_1:=\sup_{p\in \mathcal{P}_{L^\infty}^{\alpha,\beta}} \, \nu_1(p)\,.
	 	\end{equation}
	 	We call \textit{optimal pair} for \eqref{CP2} a couple $(\widehat{p},\theta^{\widehat p}_1)$ such that $\widehat p$ achieves the supremum in \eqref{CP2} and $\theta^{\widehat p}_1$ is an eigenfunction of $\nu_1(\widehat p)$.  
	 	In \cite{befa} the following result is proved.
	 	\begin{proposition}\label{thm-nu}
	 		\cite{befa}	Problem \eqref{CP2} admits an optimal pair $(\widehat{p},\theta^{\widehat p}_1) \in \mathcal{P}_{L^\infty}^{\alpha,\beta}\times H^2_\mathcal{O}(\Omega)$.
	 		Furthermore, $\theta^{\widehat p}_1$ and $\widehat{p}$ are related as follows
	 		\begin{equation*}\label{pnu}
	 		\widehat{p}(x,y) = \beta \chi_{ \widehat S} (x,y)+ \alpha \chi_{\Omega \setminus \widehat S}(x,y)\,\quad \text{for a.e. } (x,y)\in \Omega\,,
	 		\end{equation*}
	 		where $\widehat S = \{ (x,y)\in \Omega\,:\,(\theta^{\widehat p}_1)^2(x,y) \leq \widehat t \}$ for some $\widehat t> 0$ such that $|\widehat S|=\frac{1-\alpha}{\beta-\alpha}\,|\Omega|$.
	 	\end{proposition}
	 	Since we do not know explicitly $\theta^{\widehat p}_1$, the function $\theta^{\widehat p}_1$ is replaced by the torsional eigenfunction $\theta^1_1(x,y)$ of \eqref{weighteig} with $p\equiv1$, i.e. an eigenfunction corresponding to $\nu_1(1)$. This is explicitly known, see \cite{fergaz}; for details on this choice see \cite{befa}. Therefore we consider
	 	$$
	 	p^*(x,y):= \beta \chi_{S^*}(x,y)+ \alpha \chi_{\Omega \setminus S^*}(x,y) \,\quad \text{for a.e. }(x,y)\in \Omega\,,
	 	$$
	 	where $S^* := \{ (x,y)\in \Omega\,:\, (\theta^1_{1})^2(x,y) \leq t^* \}$ for $t^*> 0$ such that $|S^*|=\frac{1-\alpha}{\beta-\alpha}|\Omega|$.
	 	\vspace{2mm}
	 	\item \fbox{$\breve p(y)$}\hfill\includegraphics[scale=0.22]{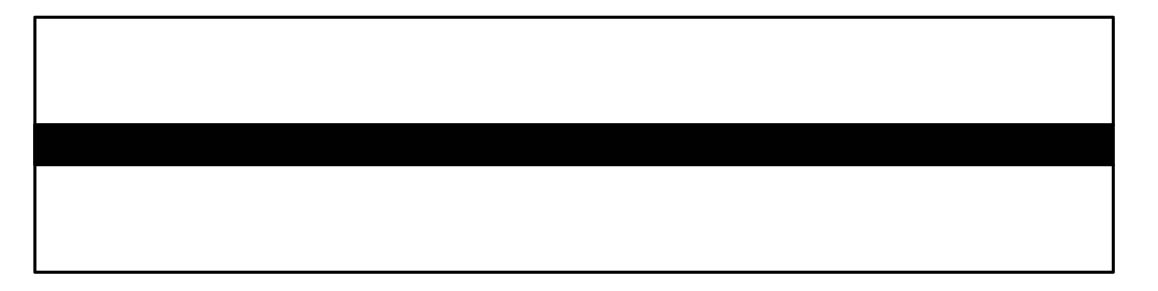}\vspace{2mm} \\
	 	In order to find a reinforce more suitable for manufacturing, inspired by $p^*(x,y)$, we consider a weight depending only on $y$ and concentrated around the mid-line $y=0$, i.e. 
	 	\begin{equation*}
	 	\breve p(x,y)=\breve p(y):=\beta \chi_{\breve I}(y)+\alpha \chi_{(-\ell,\ell)\setminus \breve I}(y)\,\quad \text{for a.e. } (x,y)\in \Omega\,,
	 	\end{equation*}
	 	where $\breve I:=\big(-\frac{\ell(\beta-1)}{\beta-\alpha},\frac{\ell(\beta-1)}{\beta-\alpha}\big)$.
	 	\vspace{2mm}
	 	\item\fbox{$\overline p_{i}(x)$, $i\in \mathbb{N}^+$}\hfill\includegraphics[scale=0.22]{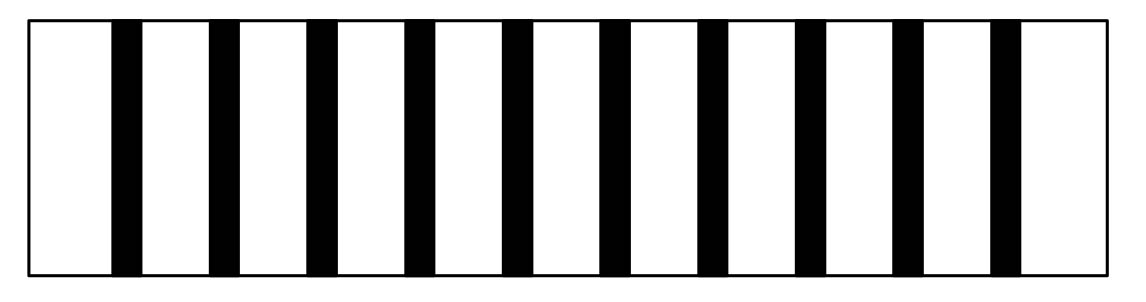}\vspace{2mm} \\
	 	The reasons of this choice are quite involved. We give here only the main idea and for details we refer to \cite{befa}.\par  
	 	For $i\in \N^+$, we set the minimum problem
	 	\begin{equation}\label{CP}
	 	\mu^{\alpha,\beta}_i :=\inf_{p \in\mathcal{P}_{L^\infty}^{\alpha,\beta}} \, \mu_i(p)\,,
	 	\end{equation}
	 	where $\mu_i(p)$ is the $i$-th longitudinal eigenvalue of \eq{weighteig}.
	 	We call \textit{optimal pair} for \eq{CP} a couple $(\overline{p}_i,z^{\overline p_i}_{i})$ such that $\overline{p}_i$ achieves the infimum in \eqref{CP} and $z^{\overline p_i}_{i}$ is an eigenfunction of $\mu_i(\overline{p}_i)$.
	 	In \cite[Theorem 3.2]{befafega} the following result is proved.
	 	
	 	\begin{proposition}\label{thm-mu1} \cite{befafega}
	 		Set $i=1$, then problem \eqref{CP} admits an optimal pair $(\overline{p}_1,z^{\overline p_1}_{1}) \in \mathcal{P}_{L^\infty}^{\alpha,\beta}\times H^2_\mathcal{E}(\Omega)$.
	 		Furthermore, $z^{\overline p_1}_{1} $ and $\overline p_1$ are related as follows
	 		\begin{equation*}\label{p1}
	 		\overline{p}_1(x,y) = \alpha \chi_{ S_1} (x,y)+ \beta \chi_{\Omega \setminus S_1}(x,y)\,\quad \text{for a.e. } (x,y)\in \Omega\,,
	 		\end{equation*}
	 		where $ S_1 = \{ (x,y)\in \Omega\,:\,(z^{\overline p_1}_{1})^2(x,y) \leq t_1 \}$ for some $ t_1> 0$ such that $| S_1|=\frac{\beta-1}{\beta-\alpha}\,|\Omega|$.
	 	\end{proposition}
	 	
	 	Things become more involved for higher longitudinal eigenvalues and we do not find an analytical expression as for $i=1$. Focusing on upper bounds for $\mu_i(p)$, see \cite{befa}, we propose the following approximated optimal weight for $\mu_i^{\alpha, \beta}$:
	 	\begin{equation*}\label{pj}
	 	\overline p_i(x,y)=\overline p_i(x):=\beta \chi_{I_i}(x)+\alpha \chi_{(0,\pi)\setminus I_i}(x)\\,\quad \text{for a.e. } (x,y)\in \Omega\,,
	 	\end{equation*}
	 	where $I_i:=\displaystyle{\bigcup_{h=1}^i}\bigg(\frac{\pi}{2i}(2h-1)-\frac{\pi}{i} \frac{(1-\alpha)}{2(\beta-\alpha)},\,\frac{\pi}{2i}(2h-1)+\frac{\pi}{i} \frac{(1-\alpha)}{2(\beta-\alpha)}\bigg)$.
	 	\vspace{2mm}
	 	\item\fbox{$\overline{\overline p}(x)$}\hfill\includegraphics[scale=0.22]{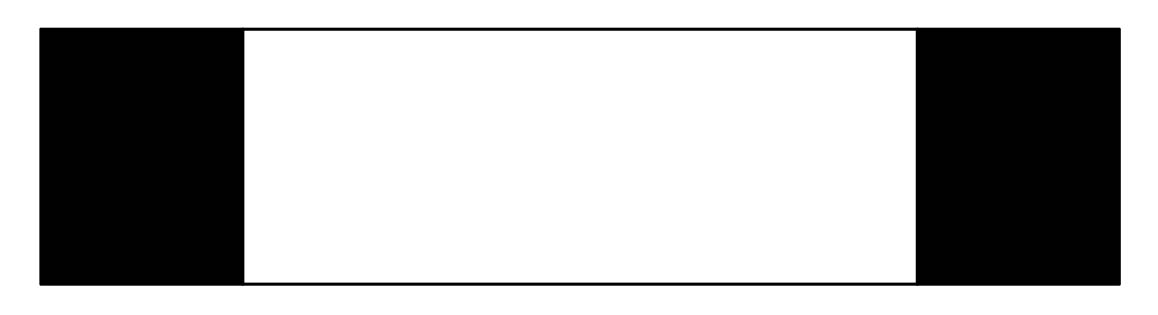}\vspace{2mm} \\
	 	We consider a weight concentrated near the short edges of the plate:
	 	$$\overline{\overline p}(x,y)=\overline{\overline p}(x):=\alpha \chi_{I}(x)+\beta \chi_{(0,\pi)\setminus I}(x) \,\quad \text{for a.e. }(x,y)\in \Omega\,,$$
	 	where $I:=\big(\frac{\pi}{2}-\frac{\pi(\beta-1)}{2(\beta-\alpha)},\frac{\pi}{2}+\frac{\pi(\beta-1)}{2(\beta-\alpha)}\big).$
	 	This weight seems to be simple for manufacturing and reasonable in order to increase $\mathcal{R}$.
	 \end{enumerate}
 We denote by 
 $$
 \widehat P_{\alpha,\beta}:=\{p\in \mathcal{P}_{L^\infty}^{\alpha,\beta}: p(x,y)\text{ coincides with } 1\text{ or } p^*(x,y)\text{ or } \breve p(y) \text{ or } \overline p_{10}(x) \text{ or }\overline{\overline p}(x) \quad \forall(x,y)\in\Omega\};
 $$
 we shall explain in the next section why we are interested in $\overline p_{10}(x)$ in the fourth case.
 \section{$L^2(\Omega)$ external forcing terms}\label{L2sec}
 When $f\in \mathcal{F}_{L^2}$ it is possible to obtain more information on the solution of \eq{eigenweak} and, in turn, on the gap function. In this case we expand $u$ in Fourier series, adopting an orthonormal basis of $L^2_p$ composed by the eigenfunctions of \eq{weighteig}. In Section \ref{proof} we prove the following result. 
 \begin{proposition}\label{gap}
 	For $m\in\mathbb{N}^+$,	we denote by $\nu_{m}(p)$ and $\mu_m(p)$ the eigenvalues of \eq{weighteig} and, respectively, $\theta^p_{m}(x,y)$ and $z^p_{m}(x,y)$ the corresponding normalized eigenfunctions, see \eqref{normalization}.\par
 	If $f\in \mathcal{F}_{L^2}$  and $p\in \mathcal{P}_{L^\infty}^{\alpha,\beta}$ then the unique solution of \eq{eigenweak} reads
 	\begin{equation}\label{fourieru}
 	u(x,y)=\sum_{m=1}^{\infty}\left[\dfrac{a_m}{\nu_m(p)}\theta^p_{m}(x,y)+\dfrac{b_m}{\mu_m(p)}z^p_{m}(x,y)\right]
 	\end{equation}
 	and 
 	\begin{equation}\label{Gap}
 	\mathcal{G}_{f,p}(x)=2\sum\limits_{m=1}^\infty\dfrac{a_m}{\nu_m(p)}\theta^p_m(x,\ell)\qquad \forall x\in[0,\pi],
 	\end{equation}
 	where 
 	$$
 	a_m:=\int_{\Omega} pf\,\theta^p_{m}\,dxdy\qquad b_m:=\int_{\Omega} pf\,z^p_{m}\,dxdy.
 	$$
 	If $f\in \mathcal{F}_{L^2}$  and $f(x,-y)=-f(x,y)$ a.e. in $\Omega$ then $u(x,y)=\sum\limits_{m=1}^{\infty}\dfrac{a_m}{\nu_m(p)}\theta^p_{m}(x,y)$.
 \end{proposition}
 
Driven by Proposition \ref{symteo}, we shall consider $y$-odd forcing terms; in \cite{antunes} the authors conjectured as worst forcing term 
 $$
 f_0(x,y)=
 \begin{cases}
 1\quad &y\in[0,\ell]\\
 -1\quad &y\in[-\ell,0).
 \end{cases}
 $$
 Since $\|f_0\|_2=\sqrt{|\Omega|}$ and we are interested in $f\in\mathcal{F}_{L^2}$, we normalize $f_0$, i.e.
  $$
\overline f_0(x,y)=
 \begin{cases}
 \dfrac{1}{\sqrt{|\Omega|}}\quad &y\in[0,\ell]\\
 \dfrac{-1}{\sqrt{|\Omega|}}\quad &y\in[-\ell,0).
 \end{cases}
 $$
We refer to Table \ref{tab1} for numerical results about $\overline f_0$.\par 
  A physical interesting case is when $f$ is in resonance with the structure, i.e. when $f$ is a multiple of an eigenfunction of \eq{weighteig}. The case in which $f$ is proportional to a longitudinal mode is not interesting from our point of view since the gap function vanishes. Hence, we consider $f$ proportional to the $j$-th torsional mode, i.e.
 $$
 f_j(x,y)=\theta^p_j(x,y);
 $$ 
since $\|f_j\|_2\neq 1$, we consider $\overline f_j(x,y)=\theta^p_j(x,y)/\|\theta^p_j\|_2$ so that $\overline f_j\in\mathcal{F}_{L^2}$ for all $j\in \mathbb{N}^+$.
 Trough Proposition \ref{gap}, we readily obtain
 $$a_m=\begin{cases}
 1/\|\theta^p_j\|_2\quad &m=j\\
 0\quad&m\neq j
 \end{cases}\qquad\quad  u(x,y)=\dfrac{\theta_j^p(x,y)}{\nu_j(p)\|\theta^p_j\|_2}\qquad\quad  \mathcal{G}_{f_j,p}(x)=2\dfrac{\theta_j^p(x,\ell)}{\nu_j(p)\|\theta^p_j\|_2}.$$
 
 We provide now some numerical results considering a narrow plate, as it may be the deck of a suspension bridge, composed by typical materials adopted for these structures, i.e.
 \begin{equation}\label{numpar}
 	\ell=\dfrac{\pi}{150}\qquad \sigma=0.2,
 \end{equation} 
 for details see \cite{bfg,crfaga,falocchi}. We point out that with these parameters the eigenvalues of the homogeneous plate ($p\equiv 1$) are ordered in the following sequence $$\mu_1(1)<...<\mu_{10}(1)<\nu_{1}(1)<\mu_{11}(1)<...$$
 Hence, the longitudinal eigenvalue closest to the first torsional from below is $\mu_{10}(1)$; for this reason we consider $p\in\widehat P_{\alpha,\beta}$ fixing $i=10$ for the fourth reinforce $\overline p_{10}$ proposed in Section \ref{weightsec}. 
  	\begin{table}[h]\centering
 	\scalebox{0.9}{	\begin{tabular}{|c||c|c|c|c|c|}
 			\hline
 			&$p\equiv1$&$p^*(x,y)$&$\breve p(y)$&$\overline p_{10}(x)$&$\overline{\overline p}(x)$\\
 			&\includegraphics[scale=0.22]{caso1.jpg}&\includegraphics[scale=0.22]{caso3.jpg}&\includegraphics[scale=0.22]{caso4.jpg}&\includegraphics[scale=0.22]{caso2.jpg}&\includegraphics[scale=0.22]{caso5.jpg}\\
 			\hline
 			\hline
 			$\nu_1(p)\cdot 10^{-4}$&1.09&\bf 1.98&1.75&1.09&1.56\\
 			\hline
 			$\nu_2(p)\cdot 10^{-4}$&4.38& 6.88&\bf7.01&4.37&4.14\\
 			\hline
 			\hline
 			$\mathcal{G}^\infty_{\overline f_0,p}\cdot10^{4}$&9.32&\bf 6.09&6.99&9.32&7.00\\
 			\hline
 			$\mathcal{G}^\infty_{\overline f_1,p}\cdot10^{4}$&1.23$\cdot10$&\bf6.74&7.71&1.23$\cdot10$&8.21\\
 			\hline
 			$\mathcal{G}^\infty_{\overline f_2,p}\cdot10^4$&3.08&\bf1.93&\bf 1.93&3.11&3.38\\
 			\hline
 	\end{tabular}}
 	
 	\vspace{3mm}
 	\caption{The first torsional weighted eigenvalues $\nu_{1}(p)$, $\nu_{2}(p)$ and $\mathcal{G}^\infty_{f,p}$ defined in \eqref{gapmax}, assuming \eqref{numpar}-\eqref{alphabeta} and $N=30$.}
 	\label{tab1}
 \end{table}
 On the choice of the values $0<\alpha<1<\beta$ related to the family $\mathcal{P}_{L^\infty}^{\alpha,\beta}$, for the applicative purpose we may strengthen the plate with steel and  we may consider the other material composed by a mixture of steel and concrete; therefore, the denser material has approximately triple density with respect to the weaker. Thus, we assume 
 \begin{equation}\label{alphabeta}
\alpha=0.5\qquad \beta=1.5.
 \end{equation}
 \begin{figure}[!hbt]
 	\centering
 	{\includegraphics[width=8.1cm]{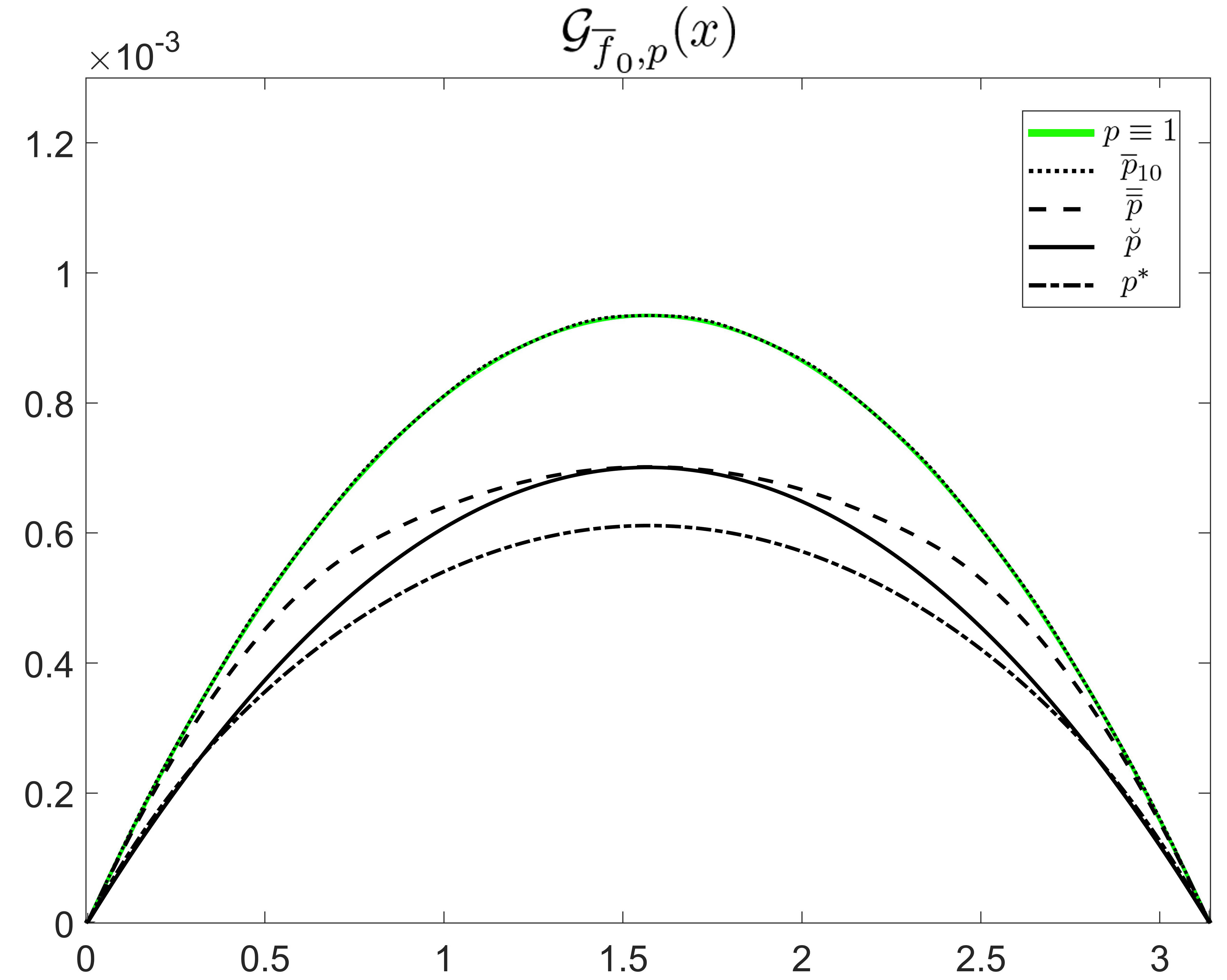}}\hspace{7mm}{\includegraphics[width=8cm]{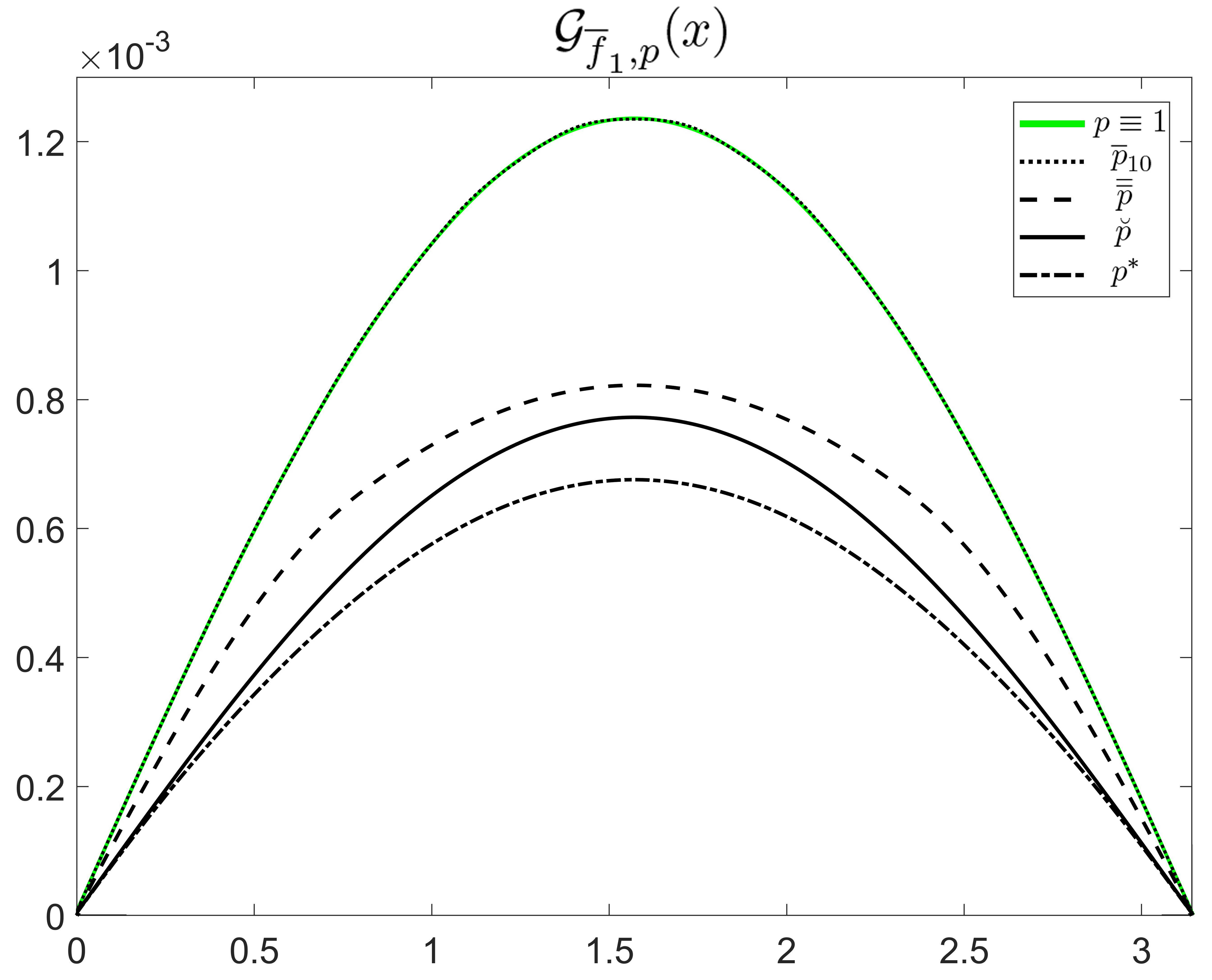}}
 	\caption{Plots of the gap functions $\mathcal{G}_{\overline f_0,p}(x)$ and $\mathcal{G}_{\overline f_1,p}(x)$ for $x\in[0,\pi]$, varying $p$, assuming \eqref{numpar}-\eqref{alphabeta} and $N=30$.}
 	\label{fig2}
 \end{figure}  
The numerical computation of the gap function in \eqref{Gap} is obtained truncating the Fourier series at a certain $N\geq1$, integer; we compute the weighted eigenvalues and eigenfunctions of \eqref{weighteig}, exploiting the explicit information we have in the case $p\equiv 1$, see \cite{fergaz}, and adopting the same numerical procedure described in \cite{befa}.\par

In Table \ref{tab1} we present the maximum values assumed by the gap function with respect to the choice of $f\in\mathcal{F}_{L^2}$ and $p\in\widehat{P}_{\alpha,\beta}$; as one can expect, for $f=\overline f_0$ the absolute maximum is always attained in $x=\pi/2$, while for $f=\overline f_j$ is assumed where $\sin(jx)$ has stationary points; indeed, $\theta^p_j(x,\pm\ell)$ is qualitatively similar to $\pm A\sin(jx)$ ($A\in\mathbb{R}^+$, $j\in\mathbb{N}^+$), see Figure \ref{fig2}.

\begin{figure}[!hbt]
	\centering
	{\includegraphics[width=7.5cm]{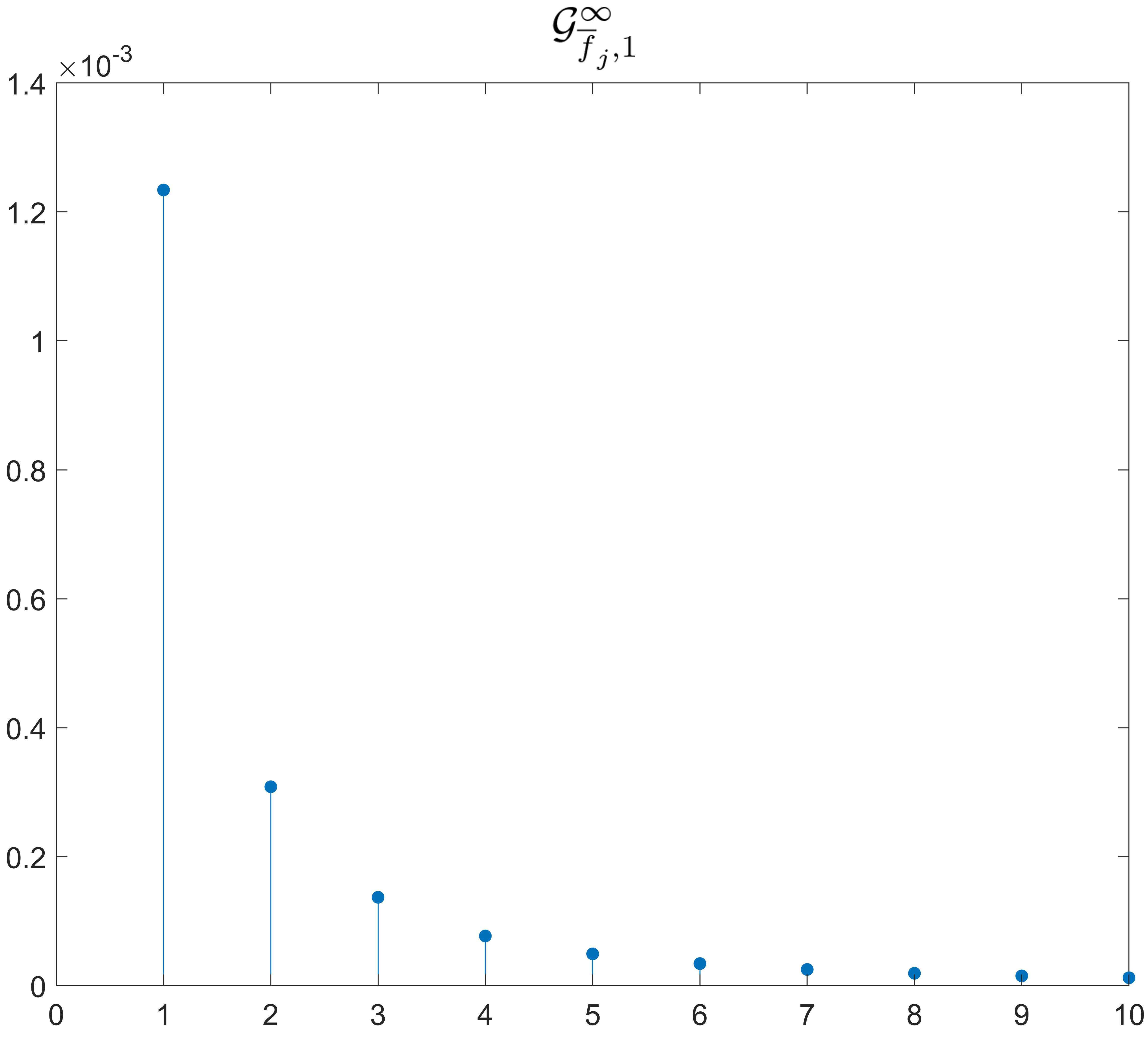}}\hspace{10mm}{\includegraphics[width=7.5cm]{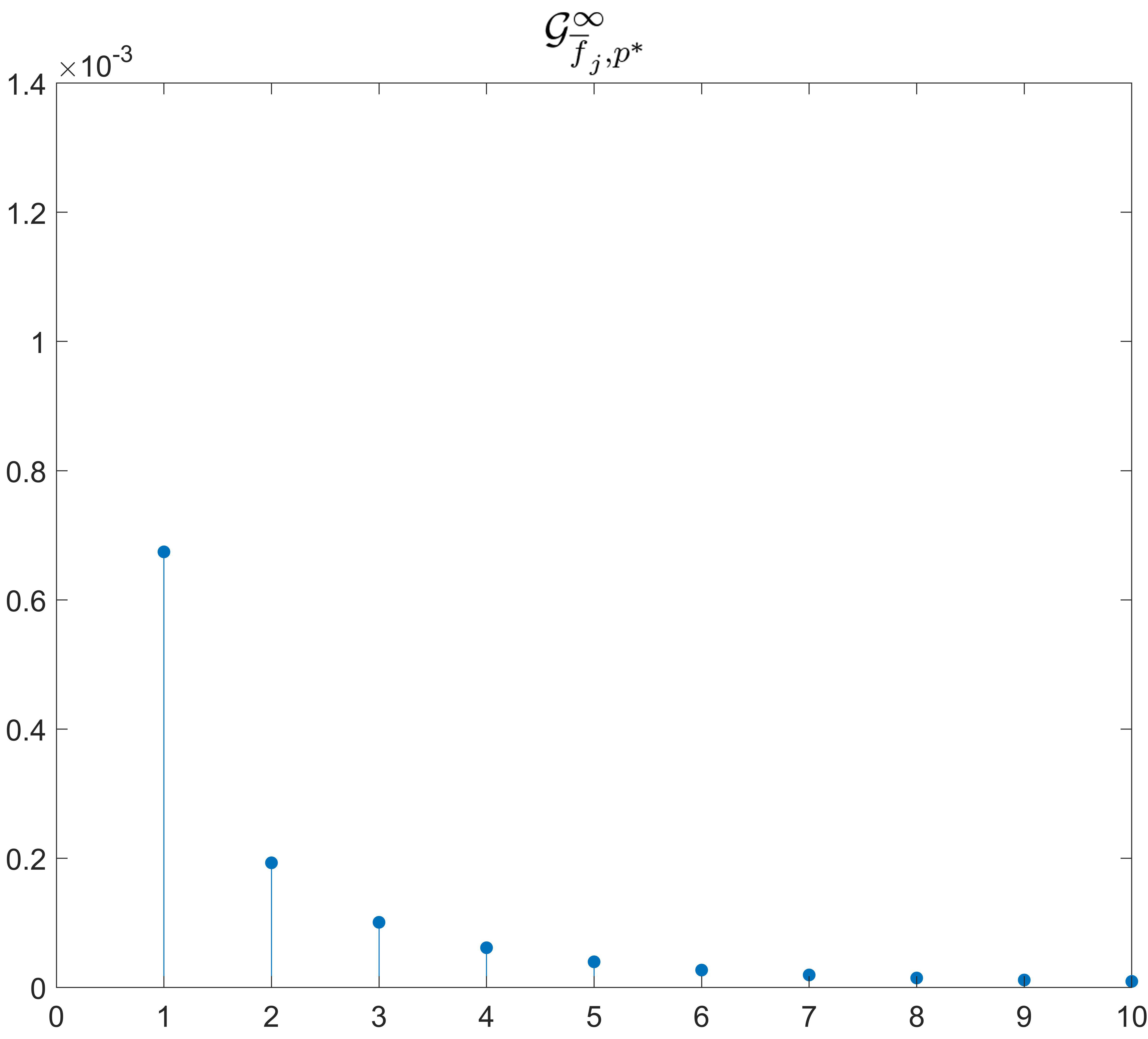}}
	\caption{Plots of $j\mapsto \mathcal{G}^\infty_{\overline f_j,1}$ and $j\mapsto \mathcal{G}^\infty_{\overline f_j,p^*}$, assuming \eqref{numpar}-\eqref{alphabeta} and $N=30$.	}
	\label{fig1}
\end{figure} 

 In Figure \ref{fig1} we plot $j\mapsto \mathcal{G}^\infty_{\overline f_j,p}$ when the plate is homogeneous and $p=p^*$; through this result we conjecture that the gap function reduces in amplitude when $\overline f_j$ is in resonance with higher torsional modes.
 
The choice to strengthen the plate with densities like $\overline p_i(x)$ ($i\in\mathbb{N}^+$) needs some remarks. In this paper we considered only the case $\overline p_{10}(x)$, because it is emblematic for all $\overline p_i(x)$; indeed, the values of $\overline p_{10}(x)$ in Table \ref{tab1} are very similar to those related to $\overline p_i(x)$ with $i=4,\dots, 15$, hence we do not show them. We point out that these reinforces are thought to reduce the $i$-th longitudinal eigenvalue, see \cite{befa}. From our analysis we observe that they are not so useful in modifying the torsional eigenvalues and in lowering the gap function; this is confirmed also by Figure \ref{fig2} where the gap function related to $\overline p_{10}(x)$ is very close to the gap function of the homogeneous plate. Numerically we observe  that this trend is less and less remarked as we increase the size of $\ell$ with respect to \eqref{numpar}. Hence, for $\ell\gg \frac{\pi}{150}$, e.g. $\ell=\frac{\pi}{15}$, it is possible that weights as $\overline p_i(x)$ ($i\in\mathbb{N}^+$) play a role in the torsional performance of the plate, but this overcomes our applicative purposes.

 The worst situation among the tested external forces appears when $f=\overline f_1$ followed by $f=\overline f_0$; this suggests that the forces $f\in\mathcal{F}_{L^2}$ which maintain the same (and opposite) sign along the two free edges of the plate seem to be the candidate solutions of \eqref{max}. Among the weight considered, the possible optimal reinforces of \eqref{min} are $p^*(x,y)$ or $\breve p(y)$, see Figure \ref{fig2}. The weight $p^*(x,y)$ provides very good results for our aims, while $\breve p(y)$ is more suitable to maximize the second torsional eigenvalue; this is also confirmed by the value of $\mathcal{G}^\infty_{\overline f_2,p}$, i.e. the maximum of the gap function when $f$ is in resonance with the second torsional weighted eigenfunction. 
 In general, this agrees with the results obtained in \cite{befa}, in which the problem is dealt with a different point of view, based on the maximization of the eigenvalues ratio $\mathcal{R}$ in \eqref{opt_intro}.

 \section{Proofs}\label{proof}
 \subsection{Proof of Theorem \ref{existence}}
Fixed $p\in\mathcal{P}_{W}^{\alpha,\beta}$ with $0<\alpha<1<\beta$, we prove the continuity of the map $f\mapsto\mathcal{G}^\infty_{f,p}$ in the following lemma.

\begin{lemma}\label{continuityV}
	Let $(V,W)$ the couple of functional spaces defined respectively in \eqref{casi}-i) or in \eqref{casi}-ii).
	The map $\mathcal{G}^\infty_{f,p}:V\rightarrow [0,+\infty)$ is continuous when $V$ is endowed with the weak* topology.
\end{lemma} 
\begin{proof}
		Let $\{f_n\}_n\subset V$ be such that $f_n\overset{\ast}{\rightharpoonup} f$ in $V$ for $n\rightarrow +\infty$. Denoting by $u_n$ the solution of \eq{eigenweak} corresponding to $f_n$, we have 
	\begin{equation}\label{weaknV}
	(u_n,v)_{H^2_*} =\langle pf_n,v\rangle \qquad\forall v\in H^2_*(\Omega);
	\end{equation}
	since $f_n\overset{\ast}{\rightharpoonup} f$ in $V$, its $V$ norm is bounded, then the above equality with $v=u_n\in H^2_*(\Omega)\subset C^0(\overline\Omega)$ gives respectively in the cases \eqref{casi}-i) and \eqref{casi}-ii)
	\begin{equation}
	\begin{split}
	&i)\qquad\|u_n\|^2_{H^2_*}=\bigg|\int_{\Omega}f_n\,pu_n\,dxdy\bigg|\leq \beta\int_{\Omega}|f_nu_n|\,dxdy\leq \beta\|f_n\|_q\|u_n\|_{q'}\leq C_3\|u_n\|_{H^2_*},\\
	&ii)\qquad\|u_n\|^2_{H^2_*}=\big|\langle pf_n,u_n\rangle\big|=\big|\langle f_n,pu_n\rangle\big|\leq  \|f_n\|_{H^{-2}_*}\|pu_n\|_{H^2_*}\leq C_4\|u_n\|_{H^2_*},
	\end{split}
	\end{equation}
in which  in the last inequality  we used that $H^2_*(\Omega)$ is a Banach algebra.
Therefore $\|u_n\|_{H^2_*}\leq C$ for some $ C>0$; thus we obtain, up to a subsequence, $u_n\rightharpoonup \overline u$ in $H^2_*(\Omega)$. Denoting by $V'$ the dual space of $V$, we get $pv\in V'$; hence we pass to the limit \eq{weaknV}
	\begin{equation*}
	(\overline u,v)_{H^2_*} =\langle f,pv\rangle \qquad\forall v\in H^2_*(\Omega),
	\end{equation*}
	obtaining by the uniqueness that $\overline u$ is the weak solution of \eqref{eigenweak}.\par 
	The embedding $H^2_*(\Omega)\subset C^0(\overline \Omega)$ is compact, therefore $u_n\rightarrow \overline u$ in $C^0(\overline \Omega)$, implying that the gap function $\mathcal{G}_{f_n,p}(x)$ converges uniformly to $\mathcal{G}_{f,p}(x)$ as $n\rightarrow+\infty$ for all $x\in[0,\pi]$. Therefore $\mathcal{G}^\infty_{f_n,p}\rightarrow \mathcal{G}^\infty_{f,p}$ as $n\rightarrow +\infty$.
\end{proof}
\begin{proof}[Proof of Theorem \ref{existence} completed]
	Let $p\in \mathcal{P}_W^{\alpha,\beta}$ fixed and $\{f_n\}\subset \mathcal{F}_{V}$ a maximizing sequence for \eq{max}; 
	since $\|f_n\|_{V}=1$, we have, up to a subsequence, $f_n\overset{\ast}{\rightharpoonup}\overline f$ in $V$. By the lower semi continuity of the norms we have $\|\overline f\|_V\leq \|f_n\|_V=1$. Through Lemma \ref{continuityV} we obtain $$\max\limits_{f\in\mathcal{F}_{V}}\mathcal{G}^\infty_{f,p}=\mathcal{G}^\infty_{\overline f,p};$$ we prove that $\|\overline f\|_V=1$.  For contradiction we suppose $\|\overline f\|_{V}<1$; hence, we set $\widehat{f}=\overline f/\|\overline f\|_{V}$ and by linearity we obtain $\mathcal{G}^\infty_{\widehat{f},p}=\mathcal{G}^\infty_{\overline f,p}/\|\overline f\|_{V}>\mathcal{G}^\infty_{\overline f,p}$. This is absurd.
\end{proof}

\subsection{Proof of Theorem \ref{existence2}}
In the proof we shall use the compactness of the set $\mathcal{P}_{W}^{\alpha,\beta}$; if $W=L^\infty(\Omega)$ the set $\mathcal{P}_{L^\infty}^{\alpha,\beta}$ is compact for the $L^\infty$ weak* topology, see \cite[Lemma 5.2]{befa}. If $W=H^2(\Omega)$ we prove the following result.
\begin{lemma}\label{compactness}
	The set $\mathcal{P}_{H^2}^{\alpha,\beta}$ with $0<\alpha<1<\beta$ is compact for the $H^2$ weak topology.
\end{lemma}
\begin{proof}
	Let $\{ p_n \}_n \subset \mathcal{P}_{H^2}^{\alpha,\beta}$,
	then by definition $\|p_n\|_{H^2}\leq\kappa\sqrt{|\Omega|}$, hence, up to a subsequence, we have $p_n\rightharpoonup \overline p$ in $H^2(\Omega)$ (as $n\rightarrow +\infty$) for some $\overline p\in H^2(\Omega)$ and $$\|\overline p\|_{H^2}\leq \liminf\limits_{n\rightarrow+\infty}\|p_n\|_{H^2}\leq\kappa\sqrt{|\Omega|}; $$ due to the compact embedding $H^2(\Omega) \subset C^0(\overline\Omega)$, we obtain $
	p_n\rightarrow \overline p$ uniformly as $n \rightarrow \infty$.  This implies $\alpha\leq \overline p\leq\beta$ and $\overline p(x,-y)=\overline p(x,y)$ for all $(x,y)\in \Omega$; moreover, passing the limit under the integral, we obtain $|\Omega|=\int_{\Omega}p_n\,\,dx\,dy\rightarrow\int_{\Omega}\overline p\,dx\,dy$, implying $\int_{\Omega}\overline p\,dx\,dy=|\Omega|$.  \par  
	Therefore the limit point $\overline p\in\mathcal{P}_{H^2}^{\alpha,\beta}$ and $\mathcal{P}_{H^2}^{\alpha,\beta}$ is compact for the $H^2$ weak topology.	  
\end{proof}

Fixed $f\in\mathcal{F}_{V}$, we endow the spaces 
\begin{equation}\label{topology}
\begin{split}
&i)\qquad L^\infty(\Omega) \quad\text{with the weak* topology},\\
&ii)\qquad H^2(\Omega) \quad\text{with the weak topology}
\end{split}
\end{equation}
and we prove the continuity of the map $p\mapsto\mathcal{G}^\infty_{p}$ in the next lemma.
\begin{lemma}\label{continuityLp2}
		Let $(V,W)$ the couple of functional spaces defined respectively in \eqref{casi}-i) or in \eqref{casi}-ii).
	The map $\mathcal{G}^\infty_{p}:\mathcal{P}_{W}^{\alpha,\beta}\rightarrow [0,+\infty)$ is continuous when $W$ is endowed with the proper topology in \eqref{topology}.
\end{lemma} 
\begin{proof}
Let $\{p_n\}_n\subset \mathcal{P}_{W}^{\alpha,\beta}$ be such that
\begin{equation*}
\begin{split}
i)\qquad &\text{if } \hspace{1mm} W=L^\infty(\Omega)\qquad p_n\overset{\ast}{\rightharpoonup} p \qquad\text{  in  }\quad L^\infty(\Omega)\\
ii) \qquad &\text{if } \hspace{1mm} W=H^2(\Omega)\qquad\, p_n\rightharpoonup p\qquad\text{ in } \quad H^2(\Omega)
\end{split}	
\end{equation*} 
for $n\rightarrow +\infty$; since $\mathcal{P}_{W}^{\alpha,\beta}$ is compact for the respective topology \eqref{topology}, then $p\in \mathcal{P}_{W}^{\alpha,\beta}$.\par 
	We denote by $u_n$ the solution of \eq{eigenweak} corresponding to $p_n$ and we get 
	\begin{equation}\label{weakn2}
	(u_n,v)_{H^2_*} =\langle p_nf,v\rangle \qquad\forall v\in H^2_*(\Omega);
	\end{equation}
 the above equality with $v=u_n\in H^2_*(\Omega)\subset C^0(\overline\Omega)$ gives respectively in the cases \eqref{casi}-i) and \eqref{casi}-ii)
 \begin{equation}
 \begin{split}
 &i)\qquad\	\|u_n\|^2_{H^2_*}=\bigg|\int_{\Omega}p_n\,fu_n\,dxdy\bigg|\leq \|p_n\|_\infty\|fu_n\|_1\leq \beta\|f\|_q\|u_n\|_{q'}\leq C_5\|u_n\|_{H^2_*},\\
 &ii)\qquad\|u_n\|^2_{H^2_*}=\big|\langle f,p_nu_n\rangle\big|\leq \|f\|_{H^{-2}_*}\|p_nu_n\|_{H^2_*}\leq C_6\|u_n\|_{H^2_*},
 \end{split}
 \end{equation}
	in which, in the last inequality we use that $H^2_*(\Omega)$ is a Banach algebra, \eqref{defdual} and $p_n\rightharpoonup p$ in $H^2(\Omega)$. This implies $\|u_n\|_{H^2_*}\leq \overline C$ for some $\overline C>0$; thus we get, up to a subsequence, $u_n\rightharpoonup \overline u$ in $H^2_*(\Omega)$ and we pass to the limit \eq{weakn2}
	\begin{equation*}
	(\overline u,v)_{H^2_*} =\langle f,pv\rangle \qquad\forall v\in H^2_*(\Omega),
	\end{equation*}
	obtaining by the uniqueness that $\overline u$ is the weak solution of \eqref{eigenweak}.\par 
	As in Lemma \ref{continuityV} we use the compact embedding $H^2_*(\Omega)\subset C^0(\overline \Omega)$, implying that the gap function $\mathcal{G}_{p_n}(x)$ converges uniformly to $\mathcal{G}_{p}(x)$ as $n\rightarrow+\infty$ for all $x\in[0,\pi]$.
\end{proof}

\begin{proof}[Proof of Theorem \ref{existence2} completed]
	By Lemma \ref{continuityLp2} we have that $p\mapsto \mathcal{G}^\infty_p$ is continuous on $\mathcal{P}_{W}^{\alpha,\beta}$ with respect to the proper topology associated to $W$ in \eqref{topology}. Moreover the set $\mathcal{P}_{W}^{\alpha,\beta}$ is compact for the correspondent topology, see \cite[Lemma 5.2]{befa} and Lemma \ref{compactness}; this readily implies the existence of the minimum \eq{min}. 
\end{proof}

\subsection{Proof of Proposition \ref{symteo}}
We follow the lines of \cite[Section 9]{bebugazu}, beginning with the second statement. 

$ii)$ Let $f\in\mathcal{F}_{H^{-2}_*}$ and $u_f\in H^2_*(\Omega)$ the solution of \eqref{eigenweak}. Being $p(x,y)$ even with respect to $y$, we use the decomposition \eqref{decomposition} and we rewrite \eqref{eigenweak} as 
\begin{equation}\label{eq1}
(u_f^o,v^o)_{H^2_*}+(u_f^e,v^e)_{H^2_*}=\langle pf^o,v^o\rangle+\langle pf^e,v^e\rangle\qquad \forall v\in H^2_*(\Omega).
\end{equation}
By \eqref{gapdef} we have $\mathcal{G}_{f,p}(x)=u^o(x,\ell)-u^o(x,-\ell)$; therefore, if $f^o=0$ then $u^o=0$ and $\mathcal{G}_{f,p}^\infty=0$, implying that $f$ cannot be a solution of \eqref{max}. Through \eqref{dualnorm} we infer the existence of $\gamma\in(0,1]$ such that $\gamma=\|f^o\|_{H^{-2}_*}\leq \|f\|_{H^{-2}_*}=1$. By linearity and \eqref{eq1} we observe that the problem $
(w,v)_{H^2_*}=\frac{1}{\gamma}\langle pf^o,v\rangle$ admits as solution $w=\frac{u^o}{\gamma}$ for all $v\in H^2_*(\Omega)$. Hence, by linearity, $\mathcal{G}_{\frac{f_0}{\gamma},p}^\infty=\frac{1}{\gamma}\mathcal{G}^\infty_{f,p}\geq\mathcal{G}^\infty_{f,p}$. Therefore for all $f\in \mathcal{F}_{H^{-2}_*}$ there exists $g\in H^{-2}_{\mathcal{O}}(\Omega)$ ($g=f^o/\gamma$) such that $\mathcal{G}^\infty_{g,p}\geq\mathcal{G}^\infty_{f,p}$, giving the thesis.

$i)$ In \cite[Lemma 9.1]{bebugazu} it is proved the following result: for $q\in[1,\infty]$, $a>0$ and $\phi\in L^q(]-a,a[)$ it holds 
\begin{equation}\label{eq2}
\|\phi^o\|_{ L^q(]-a,a[)}\leq \|\phi\|_{ L^q(]-a,a[)}.
\end{equation} 
Hence for every $q\in(1,\infty]$, \eqref{eq2} combined with the arguments used in the proof of Proposition \ref{symteo}-$ii)$ yields that $f$ odd with respect to $y$ is a maximizer.

For $q\in(1,\infty)$ we suppose, by contradiction, that $f\in\mathcal{F}_{L^q}$ is a non-odd maximizer. We point out that the inequality \eqref{eq2} is strict for  $q\in(1,\infty)$ if and only if $\phi$ is non-odd ($\phi\not\equiv\phi^o$), see again \cite[Lemma 9.1]{bebugazu} for a proof. Therefore, being $f\not\equiv f^o$, we get $\|f^o\|_q<\|f\|_q=1$; we take $\overline f=f^o/\|f^o\|_q$, so that $\|\overline f\|_q=1$. Since $f^e$ does not play a role in the gap function, we have $\mathcal{G}^\infty_{\overline f,p}=\frac{\mathcal{G}^\infty_{ f,p}}{\|f^o\|_q}>\mathcal{G}^\infty_{f,p}$. This is absurd.\hfill$\square$

\subsection{Proof of Proposition \ref{gap}}
		We choose $\{z^p_m,\theta^p_{m}\}_{m=1}^\infty$ as orthonormal basis of $L^2_p(\Omega)$ (and orthogonal basis of $H^2_*(\Omega)$).
	Since $f\in L^2(\Omega)\subset L^2_p(\Omega)$ we expand it in Fourier series
	\begin{equation*}\label{fourierf}
	f(x,y)=\sum_{m=1}^{\infty}\left[a_{m}\theta^p_{m}(x,y)+b_{m}z^p_{m}(x,y)\right]\,,
	\end{equation*}
	with $a_m,b_m \in \R$ defined as
	$$
	a_m:=\int_{\Omega} pf\,\theta^p_{m}\,dxdy\qquad b_m:=\int_{\Omega} pf\,z^p_{m}\,dxdy.
	$$
 We  write
	\begin{equation*}\label{fourier}
	u(x,y)=\sum_{m=1}^{\infty}\left[\alpha_{m}\theta^p_{m}(x,y)+\beta_{m}z^p_{m}(x,y)\right]\,,
	\end{equation*}
	where $\alpha_m,\beta_m \in \R$ are defined as
	$$
	\alpha_m:=\int_{\Omega} pu\,\theta^p_{m}\,dxdy\qquad \beta_m:=\int_{\Omega} pu\,z^p_{m}\,dxdy.
	$$
For all $m\in \mathbb{N}^+$, $z^p_{m}$ and $\theta^p_{m}$ solve:
	\begin{equation}\label{base1d}
	\begin{split}
	(z^p_{m},v)_{H^2_*}=\mu_{m}(p)\, (p\,z^p_{m}, v)_{L^2} \quad &\forall v \in H^2_*(\Omega)\,\\
	(\theta^p_{m},v)_{H^2_*}=\nu_{m}(p)\,(p\,\theta^p_{m}, v)_{L^2} \quad &\forall v \in H^2_*(\Omega)\,.
	\end{split}
	\end{equation}
	Then considering \eq{eigenweak} with $v=\theta^p_m$, $v=z^p_m$ and putting $v=u$ in \eqref{base1d} we have
	$$
	\alpha_m=\dfrac{a_m}{\nu_m(p)}\qquad \beta_m=\dfrac{b_m}{\mu_m(p)}
	$$
	and \eq{fourieru}.\par 
	Now we verify that $u(x,y)$ written in Fourier series as \eq{fourieru} belongs to $H^2_*(\Omega)$. Through \eqref{base1d} we obtain that $\bigg\{\dfrac{\theta^p_{m}}{\sqrt{\nu_m(p)}},\dfrac{z^p_m}{\sqrt{\mu_m(p)}}\bigg\}_{m=1}^\infty$
	is an orthonormal basis in $H^2_*(\Omega)$; therefore, if $\bigg\{\dfrac{a_{m}}{\sqrt{\nu_m(p)}},\dfrac{b_m}{\sqrt{\mu_m(p)}}\bigg\}_m\subset \ell^2(\mathbb{N}^+)$ we infer $u\in H^2_*(\Omega)$.
	We recall the variational representation of the eigenvalues of \eqref{weighteig}: for every $m\in \N^+$ it holds
	\begin{equation*}\label{caract1}
	\lambda_m(p)=\inf_{\substack{W_m\subset H^2_*(\Omega)\\\dim W_m=m}}\hspace{2mm}\sup_{\substack{u\in W_m\setminus\{0\}}}\frac{\|u\|^2_{H^2_*}}{\|\sqrt{p}u\|^2_2},
	\end{equation*} 
	implying the stability inequality
	$$\frac{\lambda_m(1)}{\beta}\leq \lambda_m(p)\leq \frac{\lambda_m(1)}{\alpha}\,,$$
	for every $m\in\mathbb{N}^+$. In \cite[Theorem 7.6]{fergaz} the authors find explicit bounds for the eigenvalues when the plate is homogeneous ($p\equiv 1$); in general it holds $\lambda_m(1)>(1-\sigma)^2m^4$, where $\sigma$ is the Poisson ratio, see \eqref{poisson}. Then we obtain
	$$
	\lambda_m(p)\geq \frac{\lambda_m(1)}{\beta}> \frac{(1-\sigma)^2m^4}{\beta}
	$$
	so that, being $\|f\|_2=\|\sqrt{p}\theta_m^p\|_2=1$, 
	$$
	\dfrac{|a_{m}|}{\sqrt{\nu_m(p)}}\leq\dfrac{\sqrt{\beta}\|f\|_2\|p\theta_m^p\|_2}{(1-\sigma)m^2}\leq\dfrac{\beta\|\sqrt{p}\theta_m^p\|_2}{(1-\sigma)m^2} =\dfrac{\beta}{(1-\sigma)m^2}\quad\qquad\dfrac{|b_m|}{\sqrt{\mu_m(p)}}\leq\dfrac{\beta}{(1-\sigma)m^2}
	$$
	and 
	$$
	\sum_{m=1}^\infty \frac{|a_{m}|^2}{\nu_m(p)}+\frac{|b_{m}|^2}{\mu_m(p)}\leq \frac{2\beta^2}{(1-\sigma)^2}\sum_{m=1}^\infty \frac{1}{m^4}<\infty.
	$$ 
	
	Through \eqref{fourieru} we get
	$$
	\mathcal{G}_{f,p}(x)=2\sum\limits_{m=1}^\infty\dfrac{a_m}{\nu_m(p)}\theta^p_m(x,\ell)\qquad \forall x\in[0,\pi],
	$$
	since $z^p_m(x,y)$ is $y$-even.\par If $f$ is $y$-odd then $b_m=0$.
\hfill$\square$

 \section{Conclusions}
In this paper we consider a stationary forced problem for a non-homogeneous
partially hinged rectangular plate, possibly modeling the deck of a bridge,
on which a non constant density function $p(x,y)$, embodying the non-homogeneity,
is given. The main aim is to optimize the torsional performance of the plate,
measured through the so called gap function $\mathcal{G}_{f,p}(x)$, see \eqref{gapdef}, with respect to both the weight $p$ and the external forcing term $f$; thus, we deal with the problem \eqref{gapmax} where $f$ and $p$ belong to suitable classes of functions.

 In Theorem \ref{existence} we prove the existence of an optimal force $f$ solution of \eqref{max} fixed the weight $p$ in proper functional spaces, while in the Theorem \ref{existence2} we prove the converse, i.e. the existence of an optimal density $p$ solution of \eqref{min} fixed $f$.
 Currently to find explicitly the solutions of \eqref{max} and \eqref{min} seems out of reach, therefore we propose some choices of $f$ and $p$ and we proceed numerically. 
In Proposition \ref{symteo} we prove symmetry properties on the solutions of \eqref{max}; motivated by this result, we focus on $y$-odd forces $f$ as optimal candidates of \eqref{max}.
On the other hand about the possible optimal weight functions we study five meaningful density configurations; the latter are inspired by \cite{befa}, where a similar problem in terms of weight optimization of the ratio between a torsional and a longitudinal eigenvalue  is given, see \eqref{opt_intro}.

We propose some numerical experiments when $f\in L^2(\Omega)$, because it is representative of the applications we have in mind; in this case, we state and prove Proposition \ref{gap} allowing to find a numerical scheme useful to determine the approximated solutions. Our analysis is performed imposing as parameters \eqref{numpar}-\eqref{alphabeta}, having sense in terms of civil engineering applications. We summarize our main outcomes:
\begin{itemize}
	\item[-] The forces $f\in\mathcal{F}_{L^2}$ which maintain the same (and opposite) sign along the two free edges of the plate (e.g. $\overline f_0$, $\overline f_1$) seem to be the worst in terms of torsional performance of the plate for each density function.
	\item[-] If we consider $f\propto\theta^p_j(x,y)$, i.e. proportional to the $j$-th weighted torsional eigenfunction, we get the corresponding maximum of the gap function decreasing with respect to $j$ for every density function; this means that the worst case is recorded for $j=1$, i.e. when $f$ is in resonance with the first weighted torsional eigenfunction, see Figure \ref{fig1}.
	\item[-] To improve the torsional performance of the plate, we suggest to strengthen it with a density function like $p^*(x,y)$ or $\breve p(y)$, see Section \ref{weightsec}. These weights have a strong effect in increasing the first torsional eigenvalues and they reduce the maximum of the gap function more than the others.
	\item[-] Weights as $\overline p_i(x)$ ($i\in\mathbb{N}^+$), useful to reduce the $i$-th longitudinal eigenvalue, generally do not affect the torsional response of the plate. We recorded the same behaviour as in the homogeneous plate, hence we do not suggest this kind of reinforce.   
\end{itemize}
A future development in this field is the study of the corresponding evolutionary problem. We point out that the presence of a possibly discontinuous coefficient $p(x,y)$ in front of the time-derivative term may lead to some problems, even just in writing the equation in strong form. 

Other researches may focus on other forces and density functions; is there a density function that maximizes the second torsional eigenvalue better than those in $\widehat P_{\alpha,\beta}$? How does the gap function vary in correspondence of such weight? In \cite{befa} it is  pointed out that $p^*(x,y)$ may be the candidate maximizer of the first torsional eigenvalue, but nothing is said about the maximizer of the second torsional eigenvalue. It may be interesting to study this issue, since the deck of a suspension bridge seems to be more prone to develop torsional instability on the second torsional eigenvalue, see for instance \cite{bookgaz,crfaga}.

\par\bigskip\noindent
\textbf{Acknowledgements.} The author is grateful to the anonymous referees whose relevant comments and suggestions helped in improving the exposition of the paper. The author is partially supported by the INDAM-GNAMPA 2019 grant ``Analisi spettrale per operatori ellittici con condizioni di Steklov o parzialmente incernierate'' and by the PRIN project ``Direct and inverse problems for partial differential equations: theoretical aspects and applications'' (Italy).

\end{document}